\documentclass[12pt]{amsart}
\usepackage{hyperref}
\usepackage{pb-diagram}
\usepackage{amssymb}
\usepackage{enumitem}
\input xy
\usepackage[all]{xy}
\usepackage{tikz-cd}

\newcommand{\unit}{\text{\textbf{1}}}
\newcommand{\N}{\mathcal{N}}

\newcommand{\Inv}{\operatorname{Inv}}
\newcommand{\can}{\mathbf{can}}

\newcommand{\M}{{\mathcal M}}

\renewcommand{\Vec}{\text{Vec}}

\newcommand{\cC}{{\mathcal C}}
\newcommand{\F}{{\mathbb F}}

\newcommand{\ot}{{\otimes}}
\newcommand{\GL}{{\operatorname{GL}}}
\newcommand{\FPdim}{{\operatorname{FPdim}}}

\newcommand{\ORep}{{\operatorname{O-Rep}}}

\newcommand{\Ker}{\mbox{\rm Ker\,}}

\newcommand{\ku}{{\Bbbk}}
\newcommand{\Z}{{\mathbb Z}}

\newcommand{\id}{\mbox{\rm id\,}}

\newcommand{\cB}{\mathcal{B}}
\newcommand{\cD}{\mathcal{D}}
\newcommand{\cM}{\mathcal{M}}
\newcommand{\cN}{\mathcal{N}}
\newcommand{\Vc}{\operatorname{Vec}}

\newcommand{\Obj}{\mbox{\rm Obj\,}}

\newcommand{\Rep}{\operatorname{Rep}}
\newcommand{\Irr}{\operatorname{Irr}}

\newcommand{\Hom}{\operatorname{Hom}}

\newcommand{\St}{\operatorname{Stab}}

\newcommand{\Fp}{\mathbb{F}_p}

\newcommand{\Pic}[1]{\operatorname{Pic}(#1)}

\newcommand{\Aut}[1]{\operatorname{Aut}_\otimes^{\operatorname{br}}(#1)}

\theoremstyle{plain}

\numberwithin{equation}{section}

\newtheorem{theorem}{Theorem}[section]
\newtheorem{lemma}[theorem]{Lemma}
\newtheorem{sublemma}[theorem]{Sublemma}
\newtheorem{corollary}[theorem]{Corollary}
\newtheorem{proposition}[theorem]{Proposition}

\theoremstyle{definition}
\newtheorem{definition}[theorem]{Definition}

\theoremstyle{remark}

\newtheorem{remark}[theorem]{Remark}

\theoremstyle{remark}

\author{Pavel Etingof}
\address{Department of Mathematics, Massachusetts Institute of Technology,
Cambridge, MA 02139, USA}
\email{etingof@math.mit.edu}
\author[C\'esar Galindo]{C\'esar Galindo}
\address{ Departamento de Matem\'aticas, Universidad de los Andes, Bogot\'a, Colombia}
\email{cn.galindo1116@uniandes.edu.co}
\begin{document}

\title{Reflection fusion categories}

\thanks{The authors are grateful to Daniel Nakano for discussions on cohomology of finite Chevalley groups. P.E. was supported by the NSF grant DMS-1502244. C.G would like to thank the hospitality of the Mathematics Department at MIT where part
of this work was carried out. C.G. was partially supported by Fondo de Investigaciones de la Facultad de Ciencias de
la Universidad de los Andes, Convocatoria 2018-2019 para la Financiaci\'on de Programas de Investigaci\'on, programa ''Simetr\'{i}a $T$ (inversi\'on temporal) en
categor\'{i}as de fusi\'{o}n y modulares''.}
\begin{abstract}
We introduce the notion of a {\it reflection fusion category}, which is a type of a $G$-crossed category generated by objects of Frobenius-Perron dimension $1$ and $\sqrt{p}$, where $p$ is an odd prime. We show that such categories correspond to orthogonal reflection groups over $\mathbb{F}_p$. This allows us to use the known classification of irreducible reflection groups over finite fields to classify irreducible reflection fusion categories. 
\end{abstract}


\maketitle

\section{Introduction}

In this paper we introduce the notion of a {\it reflection fusion category} and give a classification of such categories which are irreducible. Namely, a reflection fusion category is a faithful G-crossed fusion category $\cB$ whose invertible objects all sit in the trivial component $\cB_e$ and form an elementary abelian $p$-group $V$ (for an odd prime $p$) with a nondegenerate braiding (quadratic form on $V$), such that $\cB$ is tensor generated by objects of Frobenius-Perron dimension 1 and $\sqrt{p}$. The last condition is the most important one, and results in $G$ being realized as an orthogonal reflection group acting on $V$, with objects of dimension $\sqrt{p}$ corresponding to reflections. More precisely, each element $g$ of the reflection group $G$ gives rise to an invertible module category $\rho(g)=\cB_g$ over $\cB_e$ (i.e., an element of the Picard group ${\rm Pic}(\cB_e)$), and this assignment defines a morphism of 3-groups $G\to \underline{\underline{\rm Pic}}(\cB_e)$. This shows that the property of a group to be generated by reflections has a natural interpretation in the theory of tensor categories; namely, the notion of a reflection fusion category is a kind of categorification (or even 2-categorification) of the notion of a reflection group, which motivates our terminology. 

We say that a reflection fusion category $\cB$ is {\it irreducible} if $\cB_e$ has no nontrivial $G$-invariant tensor subcategories. We show that irreducible reflection categories give rise to irreducible reflection groups acting on $V$. Since irreducible orthogonal reflection groups over fields of positive characteristic have been classified, we are able to classify irreducible reflection categories. Namely, we have to analyze the obstructions $O_3(\rho)\in H^3(G,V)$ and $O_4(\rho,M)\in H^4(G,\ku^\times)$ where $M\in H^2(G,V)$ that arise in the construction of the $G$-crossed extension of $\cB_e$ (cf. \cite{ENO}, Sections 7, 8). It is easy to see that $O_3(\rho)$ always vanishes, and $O_4(\rho,M)$ also usually 
vanishes since $O_4(\rho,0)$ vanishes (as shown in the Appendix) and the group $H^2(G,V)$ is usually zero. This allows us to show that for each irreducible orthogonal reflection group there is a unique reflection category (up to twisting by an element of $H^3(G,\ku^\times)$, with a small number of exceptions. 

The organization of the paper is as follows. In Section 2 we discuss preliminaries on fusion categories and $G$-crossed categories. In Section 3 we discuss preliminaries on quadratic forms and orthogonal groups over finite fields, and give the classification of irreducible orthogonal reflection groups over $\F_p$, using the known classification of irreducible reflection groups over a finite field (up to extension of scalars). In Section 4 we introduce the notion of a reflection fusion category and show that such a category gives rise to an orthogonal reflection group over a finite field. In Section 5 we classify irreducible reflection categories using the classification of irreducible orthogonal reflection groups. Finally, in the Appendix it is shown that if $\cC$ is a pointed fusion category whose simple objects form an elementary abelian $p$-group $V$ ($p>2$) then any homomorphism $\rho: G\to O(V\oplus V^*)$ gives rise to a canonical $G$-extension of $\cC$, and if $\cC$ is equipped with a nondegenerate braiding defined by a quadratic form $Q$ on $V$ then any homomorphism 
$\rho: G\to O(V,Q)$ gives rise to a canonical $G$-crossed extension of $\cC$.

\section{Preliminaries on fusion categories and $G$-crossed  fusion categories}

\subsection{Fusion categories}

In this section we recall some basic definitions and standard notions. Much of the material here can be found in \cite{DGNO} and \cite{Book-ENO}.

Let $\ku$ be an algebraically closed field of characteristic zero. By a \emph{fusion category} we mean a $\ku$-linear semisimple rigid tensor category $\cC$ with finitely many isomorphism classes of simple objects, finite dimensional spaces of morphisms and such that $\unit$, the unit object
of $\cC$, is simple. By a
\emph{fusion subcategory} of a fusion category we  mean a full tensor abelian subcategory.

For a fusion category $\cC$ we denote by $\Irr(\cC)$ the set of isomorphism classes of simple objects in $\cC$. The cardinality of $\Irr(\cC)$ is called the \emph{rank} of $\cC$.

For a fusion category $\cC$, we will denote by $\Inv(\cC)$ the group of isomorphism classes of invertible simple objects of $\cC$.  A fusion category is called \emph{pointed} if every simple object is invertible.

The Grothendieck ring of a fusion category $\cC$ will be denoted $K_0(\cC)$. There exists
a unique ring homomorphism $\FPdim : K_0(\cC)\to \mathbb{R}$ such that $\FPdim(X) > 0$ for
any $X \in \Irr(\cC)$, see \cite[Proposition 3.3.6]{Book-ENO}. The  Frobenius-Perron dimension of a fusion category $\cC$ is defined as
\[\FPdim(\cC)= \sum_{X\in \Irr(\cC)} \FPdim(X)^2.\]

A fusion category $\cB$ is called \emph{braided} if it is endowed with a natural isomorphism 
\begin{align*}
c_{X,Y} & : X \otimes Y\to  Y \otimes X, & X,Y &\in\cC,
\end{align*}
satisfying the hexagon axioms, see \cite{JS}. 

A braided 
fusion category is called \emph{Tannakian} if $\cB \cong \Rep G$ as braided fusion categories,
for some finite group $G$, where the braiding in $\Rep G$ is the usual one
 $V \otimes W\to W \otimes V, v\otimes w\mapsto w\otimes v.$

A braided fusion category $(\cB,c)$ is called \emph{non-degenerate}
if  the unique simple object $X \in \cB$ such
that $c_{Y,X}c_{X,Y} = \id_{X\ot Y}$, for all objects $Y \in \cB$ is the unit object. For \emph{spherical} braided fusion categories, non-degeneracy
is equivalent to modularity, i.e., the invertibility of the $S$-matrix, see \cite{DGNO}.

\subsection{ $G$-crossed fusion categories and homotopy theory}\label{subsec:G-crossed}

In this subsection, we will recall the definition of $G$-crossed fusion category in the sense of Turaev \cite{MR2674592,Turaev-Arxiv} and some of the results of \cite{ENO3} about the construction of $G$-crossed  fusion categories.

\subsubsection{The Picard group of a braided fusion category}

Let $\cC$ be a fusion category. A   $\cC$-\emph{module category}  is a semisimple  $\ku$-linear  abelian category $\cM$ equipped with 
\begin{enumerate}
 \item
a $\ku$-bilinear bi-exact 
bifunctor $\ot: \cC \times \cM \to \cM$;
 \item natural associativity
and unit isomorphisms 
\begin{align*}
m_{X,Y,M}: (X\otimes Y)\ot M \to X \ot
(Y\ot M),&& \lambda_M: \mathbf{
1} \ot M\to M,    
\end{align*}
such that 
\begin{equation*}\label{left-modulecat1} m_{X, Y, Z\ot M}\; m_{X\otimes Y, Z,
M}= (\id_{X}\ot m_{Y,Z, M})\;  m_{X, Y\otimes Z, M}(a_{X,
Y, Z}\ot \id_{M}),
\end{equation*}
\begin{equation*}\label{left-modulecat2} (\id_{X}\ot l_M)m_{X,{\bf
1} ,M}= \id_{X \ot M},
\end{equation*}
\end{enumerate}
 for all $X,Y,Z \in \cC, M \in \cM$. We will denote by $\cB$-Mod, the 2-category of module categories over $\cB$, see \cite{Book-ENO} for more details.

Let $\cB$ be a braided fusion category.  In \cite{ENO3} the
tensor product $\M\boxtimes_\cB\N$ of $\cB$-module categories $\M$ and $\N$ was
defined. With this tensor product, the 2-category $\cB$-Mod has a structure of monoidal 2-category, in the sense of \cite{tricategories}. A $\cB$-module category $\M$ is called \emph{invertible} if there
exists a module category $\cN$ and equivalences
\begin{equation*}
\M\boxtimes_\cB\cN \cong \cB  \quad \text{\and} \quad
\cN\boxtimes_\cB \M \cong \cB.
\end{equation*}

We will denote by $\underline{\underline{\Pic{\cB}}}$ the monoidal 2-subcategory of $\cB$-Mod, where objects are invertible $\cB$-module categories, 1-arrows are equivalences of $\cB$-module categories and 2-arrows are $\cB$-module natural isomorphisms.  The \emph{Picard group}  $\Pic{\cB}$ is the 2-truncation of $\underline{\underline{\Pic{\cB}}}$, that is,  the set of equivalence classes of invertible $\cB$-module categories with product  $\boxtimes_\cB$.

If $\cB$ is non-degenerate, the 1-truncation $\underline{\Pic{\cB}}$ of  $\underline{\underline{\Pic{\cB}}}$ is a categorical group monoidally equivalent to $\underline{\Aut{\cB}}$, the monoidal category of braided auto-equivalences of $\cB$ and arrows monoidal natural isomorphisms, \cite[Theorem 5.2]{ENO3}. In particular, for non-degenerate braided fusion categories $\Pic{\cB}$ is isomorphic to $\Aut{\cB}$, the group of equivalence classes of braided tensor autoequivalences. 

\subsubsection{$G$-crossed  fusion categories}

Let $G$ be a finite group and let $\cC$ be a fusion category. A (faithful) $G$-grading
on $\cC$ is a decomposition $\cC = \oplus_{g\in G}\cC_g$, such that
$\cC_g \otimes \cC_h \subseteq \cC_{gh}$ and $\cC_g\neq 0$ for all $g, h \in G$.

We will denote by $\underline{G}$ the discrete monoidal category with $\Obj(\underline{G})=G$ and monoidal structure defined by the multiplication of $G$. An action of $G$ on $\cC$ is a monoidal functor $\underline{G}\to \underline{\operatorname{Aut}_{\otimes}(\cC)}$, where $\underline{\operatorname{Aut}_{\otimes}(\cC)}$ is the monoidal category of monoidal autoequivalences of $\cC$ and monoidal natural isomorphisms with tensor product given by composition of monoidal functors. 

Let $T:\underline{G}\to \underline{\operatorname{Aut}_{\otimes}(\cC)}$ be an action of a group $G$ on $\cC$. Given $X, Y\in Ob(\cC)$ and $f:X\to Y$, we will denote by $g_*(X)$ and $g_*(f)$ the image of $X$ and $f$ under the functor $T(g)$.
\begin{definition}
A (faithful) $G$-crossed fusion category is a fusion category $\cC$ equipped
with the following structures:

\begin{enumerate}
    \item[(i)] an action of $G$ on $\cC$,
    \item[(ii)] a faithful grading $\cC=\oplus_{g\in G}\cC_g$,
    \item[(iii)]  isomorphisms \begin{align*}
        c_{X,Y}:X\otimes Y\to g(Y)\otimes X, && g\in G, X\in \cC_g, Y\in \cC,
    \end{align*}natural in $X$ and $Y$. The isomorphisms $c_{X,Y}$ are called the \emph{$G$-braiding}.
\end{enumerate}This structures should satisfy the following conditions:
\begin{enumerate}
    \item[(a)] $g_*(\cC_h) \subseteq \cC_{ghg^{-1}}$, for all $g, h \in G$,
    \item[(b)] The diagrams
   \begin{equation*}\label{axiom 1 trenza}
\begin{tikzcd}
g_*(X\ot Z) \ar{dd}{\can} \ar{rrrr}{g_*(c_{X,Z})} &&&& g_*(h_*(Z)\ot X)\ar{dd}{\can}\\\\   
g_*(X)\ot g_*(Z) \ar{rrrr}{c_{g_*(X),g_*(Z)}} &&&& (ghg^{-1})_*g_*(Z)\ot g_*(X) 
\end{tikzcd}
\end{equation*}
commute  for all $X\in \cC_h, Z\in \cC, g,h\in G$.
\item[(c)]  The  diagrams

\begin{equation*}\label{trenzas G}
\begin{tikzcd}
X \ot Y\ot Z \ar{rr}{c_{X,Y\ot Z}} \ar{d}{c_{X,Y}\ot\id_{Z}} &&   g_*(Y\ot Z)\ot X \ar{d}{\can} \\
g_*(Y)\ot X\ot Z \ar{rr}{\id_{g_*(Y)}\ot c_{X,Z}}&&  g_*(Y)\ot g_*(Z) \ot X
\end{tikzcd}
\end{equation*}
commute for all $X\in \cC_g, Y,Z\in \cC$ and the diagrams 

\begin{equation*}\label{trenzas 2}
\begin{tikzcd}
X \ot Y\ot Z \ar{rr}{c_{X\ot Y, Z}} \ar{d}{\id_X\ot c_{Y,Z}} &&   (gh)_*(Z)\ot X\ot Y  \ar{d}{\can} \\
X\ot h_*(Z)\ot Y\ar{rr}{c_{X,h_*(Z)}\ot \id_Y}&&  g_*h_*(Z)\ot X \ot Y
\end{tikzcd}
\end{equation*}
commute for all $X\in \cC_g, Y\in \cC_h, Z\in \cC, g,h\in G$.
\end{enumerate}
The isomorphisms $\can$ are the natural isomorphisms constructed using the natural isomorphisms of the action of $G$ on $\cC$. 
\end{definition}

The definition of equivalence for  $G$-crossed monoidal categories can be found in \cite[Section 5.2]{GALINDO-Coherence}.

From now on by a $G$-crossed fusion category we will understand a faithful $G$-crossed fusion category.

In \cite[Theorem 7.12]{ENO3} it was proved that there is a correspondence between equivalence classes of (faithful) $G$-crossed  fusion categories $\cC$ with $\cC_e=\cB$ and equivalence classes of homomorphisms from the discrete monoidal 2-category $\underline{\underline{G}}$ to $\underline{\underline{\Pic{\cB}}}$, or equivalently homotopy classes of maps from $BG$ to $B\underline{\underline{\Pic{\cB}}}$.

In order to construct all maps from $\underline{\underline{G}}$ to $\underline{\underline{\Pic{\cB}}}$ we can use  obstruction theory, see \cite[Section 8]{ENO3}. First we fix a group homomorphism  $\rho:G\to \Pic{\cB}$. The obstruction to the existence of a  monoidal functor $\widetilde{\rho}:\underline{G}\to \underline{\Pic{\cB}}$, whose truncation is $\rho$ is given by the element $O_3(\rho)\in H^3(G,\Inv(\cB))$. In case that $O_3(\rho)$ vanishes, the equivalence classes of liftings of $\rho$ form a non-empty torsor over $H^2(G,\Inv(\cB))$. Now, once a monoidal functor $\widetilde{\rho}: \underline{G}\to \underline{\Pic{\cB}}$ is fixed, again the obstruction to the existence of a homomorphism  $\widetilde{\widetilde{\rho}}:\underline{\underline{G}}\to \underline{\underline{\Pic{\cB}}}$ whose truncation is $\widetilde{\rho}$ is given by the element $O_4(\widetilde{\rho})\in H^4(G,\ku^\times)$. In case that $O_4(\widetilde{\rho})$ vanishes, the equivalence classes of liftings of $\widetilde{\rho}$ form a non-empty torsor over $H^3(G,\ku^\times)$.

Following \cite{MR1128130}, we define for each $\omega \in Z^3(G; \ku^\times)$,  the maps

\begin{align*}
\gamma_{g,h}(x) := \frac{\omega(g,h,x)\omega(^{gh}x,g,h)}
{\omega(g,\ ^hx,h)},&&
\mu_g(x,y) := \frac{\omega(^gx,g,y)}
{\omega(^gx,^gy,g) \omega(g,x,y)}.
\end{align*}

for all $g,h,x,y \in G$.

The next proposition describes explicitly the action of $H^3(G,\ku^\times)$ on the set of $G$-crossed fusion categories.

\begin{proposition}
Let $\cB=\oplus_{g\in G}\cB_g$ be a $G$-crossed  fusion category with $G$-action \begin{align*}
    (g_*,\psi_g):\cB \to \cB, && \phi(g,h):(gh)_*\to g_*\circ h_*,
\end{align*} for all $g,h \in G$. For each $\omega \in Z^3(G,\ku^\times)$ we can define a new  $G$-crossed  fusion category $\cB^\omega$. As $G$-graded $\ku$-linear abelian category $\cB^\omega = \cB$,   associativity constraint and $G$-braiding
\begin{align*}
a^\omega_{X_g,X_h,X_k}:=\omega(g,h,k)a_{X_g,X_h,X_k}, &&  c_{X_g,X_h}^\omega:= c_{X_g,X_h},
\end{align*}
and $G$-action  $g_*:\cB\to \cB$ with the natural isomorphisms
\begin{align*}
   \psi_g^\omega(V_x\otimes V_{y})&:=\mu_g(x,y)\psi_h(V_x\otimes V_{y}),\\
\phi^\omega(g,h)(V_x)&:=\gamma_{g,h}(x)\phi(g,h)(V_x),
\end{align*}
for all $V_g\in \cB_g, V_x\in \cB_x, V_y\in \cB_y$, $g,x,y\in G$. Moreover, for $\omega, \omega' \in Z^3(G,\ku^\times)$, we have that $\cB^\omega$ is equivalent to $\cB^{\omega'}$ as $G$-crossed  fusion categories if and only if $\omega$ and $\omega'$ are cohomologous.
\end{proposition}
\begin{proof}
It is a straightforward computation that follows from the 3-cocycle condition of $\omega$.

 \end{proof}

Let $\cB$ and $\cD$ be  $G$-crossed fusion categories. We will say that $\cB$ and $\cD$ are \emph{twist equivalent} if there is $\omega \in Z^3(G,\ku^\times)$ such that $\cB^\omega$ and $\cD$ are equivalent as $G$-crossed  fusion categories.

\section{Finite reflection groups over finite fields}

A \emph{quadratic form} on a vector space $V$ over a field $\F$ is a function $Q:V\to \F$ such that 
\begin{enumerate}
    \item[(i)] $Q(\alpha v)=\alpha^2Q(v)$, \quad $v\in V, \alpha \in \F$,
\item[(ii)] $Q(v+u)-Q(v)-Q(u)$ is $\F$-bilinear in $v$ and $u$.
\end{enumerate}

If $B$ is a symmetric bilinear form on $V$, the \emph{associated quadratic} form $Q:V\to \F$ is defined by $Q(v):=B(v,v)$. If $\F$ has characteristic different from 2,  
\[
B(u,v)=\frac{1}{2}[Q(v+u)-Q(u)-Q(v)],
\] and the bilinear form $B$ is completely determined by the quadratic form $Q$. The quadratic form $Q$ is called \emph{non-degenerate} if $B$ is non-degenerate.

A (finite) \emph{quadratic space} is a pair $(V,Q)$ where $V$ is a finite dimensional vector space and $Q$ a non-degenerate quadratic form on $V$. 

A map $f:(V,Q_V)\to (W,Q_W)$ between quadratic spaces is said to be an \emph{isometric} isomorphism if

\begin{itemize}
    \item[(i)] $f:V\to W$ is a linear isomorphism,
    \item[(ii)] $Q_W(f(v))=Q_V(v)$ for all $v\in V$.
\end{itemize}
The \emph{orthogonal group} of $(V,Q)$ is the group of isometric isomorphisms from a quadratic space to itself, that is, 
\[O(V,Q):=\{f\in \GL(V): Q(f(v))=Q(v), \  \forall v\in V\}.\] 

\subsection{Quadratic spaces  over finite fields}\label{subsec:basic def quad spaces}

In this section all vector spaces will be over a finite field $\F_q$, a field of $q=p^k$ elements, $p$ an \emph{odd} prime.

A quadratic space is said to be \emph{isotropic} if there is a non-zero vector on which the form evaluates to zero. Otherwise the quadratic form is \emph{anisotropic}.

A 2-dimensional  quadratic space $L$ is called a \emph{hyperbolic plane} if it has a basis $\{u,v\}$ with $Q(u)=Q(v)=0$ and $B(u,v)=1$. 

For a fixed dimension there are exactly two isomorphism classes of  quadratic spaces, \cite[Corollary 4.10]{Grove-Book}. Representatives of each class are 
\begin{align}\label{Def Q+}
    (\F_q^n,Q_1), &&  Q_1(x_1,\ldots x_n)=\sum_{i=1}^nx_i^2,
\end{align}and 
\begin{align}\label{Def Q-}
    (\F_q^n,Q_\alpha), &&  Q_\alpha(x_1,\ldots x_n)=\alpha x_1^2+\sum_{i=2}^nx_i^2,
\end{align}where $\alpha \in \F_q^{\times}$ is a non-square element.

We write $\F^{\times 2}=\{a^2| a\in \F^\times\}$, the subgroup of all squares in $\F^\times$. If $B$ is a non-degenerate symmetric bilinear form on $V$ and $\hat{B}$ is a representing matrix, the \emph{discriminant} is defined as 
\[
\operatorname{Discr}(B):=
 \det(\hat{B})\F^{\times 2} \in \F^{\times}/\F^{\times 2}\cong \Z/2\Z. 
\]
The dimension of $V$ and the discriminant are complete invariants for  quadratic spaces.

If $(V,Q)$ is a quadratic space of odd dimension, and $d\in \F^\times$ is a non-square, the quadratic space $(V,dQ)$ has $\operatorname{Discr}(dQ)=d\operatorname{Discr}(Q)$, that is, $(V,dQ)$ is  non-isomorphic to $(V,Q)$ and 
\[O(V,Q)=O(V,dQ).\]Thus, in odd dimension the orthogonal groups will be denoted just by \[O(2n+1,q).\] 

In even dimension a  quadratic space $(V,Q)$ can be written as a direct orthogonal sum:

\[V=L_1\oplus L_2\oplus\cdots\oplus L_m\oplus W,\]where the $L_i$'s are hyperbolic planes and $W$ is the zero space or an anisotropic plane. If $W=0$,  the orthogonal group $O(V,Q)$ is denoted by $O^+(2n,q)$, otherwise  $O(V,Q)$ is denoted by $O^-(2n,q)$. The groups  $O^+(2n,q)$ and $O^{-}(2n,q)$ have different orders. 

\subsection{Orthogonal reflection groups and their twisted forms}

In this subsection we will recall some basic notions about reflection groups.

We will denote by $V$ a finite dimensional vector space over a field $\F$ of characteristic zero or $p>2$ and $B$ a non-degenerate symmetric bilinear form on $V$.

Let $a\in V$ with $B(a,a)\neq 0$. As usual, we define $r_a\in O(V,Q)$, the \emph{reflection} along $a$  by \[r_a(v)= v- 2\frac{B( v ,a )}{B( a ,a )}v, \  \  \text{ for all } v\in V,\]where $B$ is the bilinear form associated to $Q$.

\begin{definition}

Let $(V,Q)$ be a  quadratic space. A \emph{ (finite) orthogonal reflection group} in $(V,Q)$ is a (finite) subgroup of $O(V,Q)$ generated by reflections. An orthogonal reflection group $G\subset O(V,Q)$ is called irreducible if $V$ is an irreducible representation of $G$.
\end{definition}
Let $\F_q$ be a finite field, and let $L$ be an extension field of $\F_q$. If $V$ is a vector space over $\F_q$, we will denote $V\otimes_{\F_q}L$ by $V_L$. If $(V,Q)$ is a quadratic space over  $\F_q$, we will denote by $(V_L,q_L)$ the quadratic space on $V_L$ with quadratic form $q_L:=q\otimes_{\F_q}\id_L$.

For a group $G$ (not necessary finite), and a  field $\F$,  the category of $\F$-linear orthogonal representations  of $G$ will be denoted by $\ORep(G,\F)$. The objects of $\ORep(G,\F)$ are pairs $((V,Q),\rho)$, where $(V,Q)$ is a quadratic space over $\F$ and $\rho:G\to O(V,Q)$ is a group homomorphism. A morphism $f:((V,Q_V),\rho_V)\to ((W,Q_W),\rho_W)$ is a $G$-linear map $f:V\to W$ such that $Q_W\circ f= Q_V$.  

If  $\rho: G\to O(V,Q),$ is an orthogonal representation and $\F\subset L$ is a field extension, we will denote $\rho_L$ the orthogonal representation $\rho_L(g)=\rho(g)\otimes_{\F}\id_{L}$ over the quadratic space $(V_L,Q_L)$. Two orthogonal representations $\rho$ and $\rho'$ are called \emph{twisted forms} of each other if $\rho_{\overline{\F}}\cong \rho'_{\overline{\F}}$, where $\overline{\F}$ is an algebraic closure of $\F$.

\begin{proposition}\label{lemma:descent}
Every irreducible orthogonal reflection group over a finite field has a unique non-isomorphic twisted form.
\end{proposition}
\begin{proof}
It follows from \cite[Propostion 1, \S 2]{Bourbaki} that any irreducible reflection representation is absolutely irreducible, that is, $\operatorname{End}_G(V)=\F_q$. By the general theory of Galois Descent, twisted forms are classified by 
$H^1(\operatorname{Gal}({\overline{\F_q}}/\F_q),\operatorname{Aut}(\rho_{\overline{\F}}))$, see \cite[Theorem III.8.15]{MR2723693}. We have that \[\operatorname{Aut}(\rho_{\overline{\F_q}})= \{\id_{V_{\overline{F_q}}}, -\id_{V_{\overline{\F_q}}}\}\cong \Z/2,\]
hence $H^1(\operatorname{Gal}({\overline{\F_q}}/\F_q),\operatorname{Aut}(\rho_{\overline{\F}}))\cong \Z/2\Z$. 
\end{proof}
\begin{remark}
Given an absolutely irreducible orthogonal representation $(\rho,(V,Q))$ of $G$ and a non-square element $\alpha\in \F_q$, the orthogonal representation $(\rho, (V,\alpha Q))$ is a non-isomorphic twisted $\F_q$-form. In fact, since $\rho$ is absolutely irreducible, if $f:(\rho,(V,Q))\to (\rho,(V,\alpha Q))$ is an isomorphism, we must have that $f=\beta \id_V$ and $\alpha Q(v)= Q(\beta v)$ for all $v\in V$. The last equation implies that $\alpha=\beta^2$, a contradiction if $\alpha$ is a non-square element. 
\end{remark}

\subsection{Orthogonal irreducible reflection groups over $\F_p$}

For future applications, in this section we will present the classification of irreducible orthogonal reflection groups, based on the classification of irreducible reflection groups over finite fields ( see  \cite{MR603578} and\cite{Warner}) and \cite[Theorem 3.1]{MR2074400}.

\subsubsection{Reflection groups of orthogonal type}

If $(V,Q)$ is a  quadratic space over $\F_q$, the full orthogonal group $O(V,Q)$ is an irreducible reflection group, \cite[Theorem 3.20]{Artin}. For $\dim(V)>1$, the group $O(V,Q)$ has two conjugacy classes of reflections, each of which generates an irreducible reflection group. 

If $\dim(V)$ is even then these two groups are conjugate under $\GL(V)$ and we will denote them by $O^{\pm}_{i}(2n,q)\subset O^{\pm}(2n,q)$. If $\dim(V)$ is odd, the  groups are not isomorphic. One of them, $O_1(2n+1,q)$, has center $\{\id_V, -\id_V\}$, and the other one,  $O_2(2n+1,q)$, has trivial center.
\subsubsection{Reflection groups of type $\overline{A}_{n,p}$}\label{subsubsec:modular A}

The symmetric group $\mathbb{S}_{n+2}$ acts on \[V=\{(x_1,\ldots,x_{n+2})\in \F_p^{n+2}:\sum_ix_i=0\},\] by permuting the coordinates. If $p | n+2$ and $n>2$, the 1-dimensional subspace $W=\operatorname{Span}_{\F_p}\{(1,\ldots,1)\}\subset V$ is invariant and the quotient  $V/W$ is an $n$-dimensional irreducible faithful representation of $\mathbb{S}_{n+2}$. This representation of $\mathbb{S}_{n+2}$ will be called $\overline{A}_{n,p}$.

\subsubsection{Modular reduction of reflection groups of Coxeter type}

Finite reflection groups over Euclidean spaces are Coxeter groups and they arise from root systems. We will follow \cite[page 93]{Kane-Book} for the notation and realization of the irreducible root systems.

Let $G=\langle r_1,\ldots ,r_n\rangle$ be an irreducible (crystallographic) finite reflection group with root system $A_n (n\geq 1), B_n  (n\geq 2), D_n (n\geq 4), E_n$ $(n=6,7,8),$ or $F_4$. For any odd prime $p$, the reduction of $G$ mod $p$, denoted by $G_p$ is the subgroup of $\GL_n(\F_p)$ generated by the modular images of the reflections $r_i$'s. These reductions will be denoted by $A_{n,p}, B_{n,p}, E_{n,p},$ and $F_{4,p}$.

For $A_n$, if $p | n+1$, the reduction is not irreducible (see the case $\overline{A}_{n,p}$). Analogously to $\overline{A}_{n,p}$,  the reflection representation of the Weyl group of type $E_6$ over $\Bbb F_3$ has a one dimensional invariant subspace, and the quotient by this subspace is a faithful 5-dimensional irreducible representation of this group, giving a reflection group which we denote by $\overline{E}_{5,3}$. However, $\overline{E}_{5,3}\cong O_2(5,3)$, which is why we will assume $p>3$ for $E_6$. 

The root systems of $H_3$ and $H_4$ can be defined over $\mathbb{Z}[e^{\pi i /5}]$ and their  discriminants are $3-\sqrt{5}$ and $(7-\sqrt{5})/2$, respectively, see \cite[Table 1, page 33]{Humphreys-reflection-book}

Let $p$ be a prime number and $p\mathbb{Z}[e^{\pi i /5}]\subset \mathfrak{m}\subset \mathbb{Z}[e^{\pi i /5}]$ a maximal ideal, then $\mathbb{K}_p=\mathbb{Z}[e^{\pi i/5}]/\mathfrak{m}$ is a finite field of characteristic $p$.  
The reduction of $H_i$ ($i=3,4$) modulo $p$, denoted by $H_{i,p}$, is the image of $W(H_i)$ (the reflection group associated to the root system) under the map $\rho_i:W(H_i)\subset \GL_n(\mathbb{Z}[e^{\pi i /5}])\to \GL_n(\mathbb{K}_p)$. If $p>2$, then  $\rho_i$ is injective. 

It follows from the realization of  the root systems of $H_i$ ($i=3,4$) given in \cite[page 95]{Kane-Book} that a set of generating reflections is given by

\begin{align*}
    H_{3,p}&=\Big \langle r_{(\alpha,\alpha-1,-1)}, r_{(-\alpha,\alpha-1,1)}, r_{(1,-\alpha,\alpha-1)} \Big \rangle,\\
    H_{4,p}&=\Big \langle r_{(\alpha,\alpha-1,-1,0)}, r_{(-\alpha,\alpha-1,1,0)}, r_{(1,-\alpha,\alpha-1,0)}, r_{(1,0,-\alpha,\alpha-1)} \Big \rangle,
\end{align*}
where $\alpha \in \mathbb{K}_p$, $\alpha^2-3\alpha+1=0$.

\begin{proposition}
Let $p$ be an odd prime. The reflection representation of $H_{i,p}$ ($i=3,4$) can be realized over $\F_p$ if and only if $p=5$ or $p>5$ and \begin{equation}\label{condition H}
    p^2-1= 0 \mod 5.
\end{equation}
\end{proposition}
\begin{proof}
We have that
\begin{align*}
\operatorname{Tr}(r_{(\alpha,\alpha-1,-1)}r_{(-\alpha,\alpha-1,1)})&=\operatorname{Tr}(r_{(\alpha,\alpha-1,-1,0)}r_{(-\alpha,\alpha-1,1,0)})\\ &=\alpha-2.    
\end{align*}
Hence, if the reflection representation can be realized over $\F_p$, there is $\beta \in \F_p$ such that $\beta^2-\beta+1=0$. Thus, we have $p\geq 5$. For $p=5$, $\alpha=4 \in \F_5$, defines a realization of the reflection representation of $H_i$ ($i=3,4$). If $p>5$, the quadratic reciprocity implies that $x^2-x+1$ has a root in $\F_p$ if and only if  condition \eqref{condition H} holds. Conversely, let $\beta \in \Fp$, with $\beta^2-\beta+1=0$, then $\alpha=\beta+2\in \F_p$ satisfies $\alpha^2-3\alpha+1=0$, that is, $\alpha$ defines a realization of the reduction of the reflection representation of $H_i$ over $\F_p$.
\end{proof}

\begin{remark}
\begin{itemize}
\item  Let $p>5$, $p^2-1=0$ mod 5, and let $\zeta \in \F_p$ is a root of $5$. The discriminants  of the symmetric bilinear forms of the reductions of $H_3$ and  $H_4$ are  $3-\zeta$ and $(7-3\zeta)/2$, respectively. Thus, the reduction mod $p$ is non-degenerate. These reductions will be denoted by $H_{i,p}^{\zeta}$.
    
    \item The choice of the square root of $5$ in the reduction of $H_3$ affects the isomorphism class of the associated quadratic space. For instance, he discriminant of $H_{3,19}^9$ is $3$, and 3 does not have a square root modulo 19. On the other hand, the discriminant of $H_{3,19}^{10}$ is 6, and 5 is a square root of $6$ modulo 19. Thus  the reflection groups $H_{3,19}^9$ and $H_{3,19}^{10}$ are defined on non-equivalent quadratic spaces.

\item If $p=5$, the reductions $H_{i,5}$ also admit an invariant symmetric bilinear form. Moreover, $H_{3,5}\cong O_1(3,5)$ and $H_{4,5}\cong O^+(4,5)$, see \cite[0.11]{MR603578}. For this reason we have assumed $p>5$ for $H_{i,p}$ in Table \ref{tab:coxeter}.
\end{itemize}
\end{remark}

Finally, let us discuss the reductions of Coxeter groups of rank 2.

\begin{proposition}
Let $(V,Q)$ be a quadratic plane over $\F_p$ and $G\subset O(V,Q)$ an irreducible reflection group. Then $G$ is a dihedral group of order $2d$, where $d| p-1$ if $(V,Q)$ is a hyperbolic plane or $d | p+1$ if $(V,Q)$ is an anisotropic plane. 
\end{proposition}

\begin{proof}
The proposition follows from the description of the orthogonal groups of quadratic planes, see \cite{Grove-Book}. 

Namely, if $V$ is a hyperbolic plane then $SO(V,Q)\cong \F_p^\times$ and  every $f\in O(V,Q)$ with $\operatorname{det}(f)=-1$ is a reflection. Moreover,  $O(V,Q)$  is a semidirect product $SO(V,Q)\rtimes \langle r_u \rangle$ for any reflection $r_u\in O(V,Q)$. The group $O(V,Q)$ is isomorphic to the dihedral group of order $2(p-1)$.

In order to describe $O(V,Q)$ for an anisotropic plane, consider $\gamma \in \Fp$ a non square element. Let $\F_{p^2}:=\Fp(\sqrt{\gamma})$, Then $\F_{p^2}$ is a two dimensional $\Fp$-space and the norm map
\begin{align*}
N:\F_{p^2}\to \Fp,&& a+b\sqrt{\gamma} \mapsto a^2-b^2\gamma,
\end{align*}
is an anisotropic quadratic form. Each $c\in K$ with $N(c)=1$ defines a map $\rho_c\in SO(\F_{p^2},N)$ via \begin{align*}
    \rho_c:\F_{p^2}\to \F_{p^2},&& v\mapsto cv.
\end{align*} In fact, this correspondence defines an isomorphism $SO(\F_{p^2},N)\cong \ker(N|_{K^*})$. Thus $SO(\F_{p^2},N)$ is a cyclic group of order $p+1$.

The map \begin{align*}
    \tau: \F_{p^2}\to \F_{p^2}, && a+b\sqrt{\gamma}\mapsto a-b\sqrt{\gamma}
\end{align*}is an orthogonal reflection. The group $O(V,N)$  is a semidirect product $SO(V,N)\rtimes \langle \tau \rangle$. Hence $O(V,N)$ is isomorphic to the dihedral group of order $2(p+1)$.

Now the proposition follows easily from the description of the subgroups of a dihedral group.
\end{proof}
The reflection groups constructed in the previous proposition will be denoted by $I_{2,p}(d)$.

\subsubsection{Classification theorem}

\begin{theorem}\label{Th:classification irr othogonal reflection groups}
Suppose $G \subset O(V,Q)$ is a finite irreducible group generated by orthogonal reflections over the finite field $\F_p$, 
with $p$ an odd prime. Then, $G$ is a twisted form of

\begin{enumerate}[leftmargin=*,label=\rm{(\alph*)}]
\item\label{reduction mo p2} the reduction mod p of a finite irreducible Coxeter group; or

\begin{table}[!htbp]

\begin{tabular}{ |p{2cm}|p{2cm}|p{6cm}|  }
 \hline
 $G$ & $\operatorname{dim}(V)$ &  Conditions on $p$ \\
 \hline
 $A_{n,p}$  &  $n\geq 1$   & $p\not{|} \ n+1$ \\
 \hline
   $B_{n,p}$  &  $n\geq 3$   &  \\
 \hline
   $D_{n,p}$  &  $n\geq 4$   & \\
 \hline
    $E_{6,p}$  &  $6$   & $p\neq 3$\\
 \hline
   $E_{7,p}$  &  $7$   &   \\
 \hline
    $E_{8,p}$  &  $8$   &  \\
 \hline
    $F_{4,p}$  &  $4$   &   \\
 \hline
     $H_{3,p}^\zeta,$    &  $3$   & $p>5$ and $p^2-1\equiv 0 \mod 5$ \\
 \hline    $H_{4,p}^\zeta$  &  $4$   &   $p>5$ and $p^2-1\equiv 0 \mod 5$ \\
 \hline
  $H_{4,p}^\zeta$  &  $4$   &  $p>5$, $p^2-1\equiv 0 \mod 5$ \\
 \hline   $I_{2,p}(d)$  &  $2$   &   $d| p\pm 1$. If $(V,Q)$ is anisotropic $d|p+1$  and  $d|p-1$ in the other case. \\
 \hline
\end{tabular}
\caption {Irreducible Coxeter reflection groups over $\Fp$.} \label{tab:coxeter}
\end{table}

\item\label{fam Sn+2}  $\overline{A}_{n,p}\cong \mathbb{S}_{n+2}$ , when $p | (n + 2)$; or 

\item\label{orthogonal-type2} the full orthogonal group $O(V,Q)$, or its subgroups $O_1(V,Q)$ or $O_2(V,Q)$.
\end{enumerate}
\end{theorem}
\begin{proof}
Let $G\subset O(V,Q)\subset \GL(V)$ be an irreducible reflection group. Since $O(V,Q)$ contains no non-trivial transversions, the classification of irreducible reflection groups over finite fields (up to scalar extension)  can be applied, see \cite{MR603578} and \cite{Warner}. 

We need to see what groups in the main Theorem of \cite{MR603578} admit a non-zero symmetric $G$-invariant bilinear form. It follows from  \cite[Theorem 3.1]{MR2074400} that the groups $G_p(m,m,n)$, $G_p(2m,m,n)$, $[EJ_3(5)]^3$ and $[J_4(4)]^3$ do not admit non-zero invariant bilinear forms. By definition, the groups of orthogonal type have a non-zero invariant bilinear form. Under the restriction of  Table \ref{tab:coxeter}, the induced canonical bilinear forms of the reductions mod $p$ of finite Coxeter groups are non-zero. For the case $\overline{A}_{n,p}$, we will follow the  notation of Subsection \ref{subsubsec:modular A}. A basis of $V\subset \F_p^{n+2}$ is given by  $\{v_i:=\epsilon_i-\epsilon_{i+1}: 1\leq i<n+2\}$, where $\{\epsilon_j: 1\leq j\leq n+2\}$ is the canonical  orthogonal basis of $\F_p^{n+2}$. Thus $v_1+2v_2+\ldots +(n+1)v_{n+1}=0$ modulo $W:=$Span$_{\F_p}\{(1,\ldots,1)\}$ and $\{\overline{v}_i:=\epsilon_i-\epsilon_{i+1}: 1\leq i\leq n\}$ is a basis of $V/W$. Hence the discriminant of the canonical associated bilinear form on $V/W$ is $n+1=-1\in \F_p$.
\end{proof}

\section{Reflection fusion categories}

\subsection{Braided pointed fusion categories and metric groups}\label{Subsection:braided pointed fusion cat}

Let $(\cB,c)$ be a pointed braided fusion category. Then the set of isomorphism
classes of simple objects of $\cB$ forms a finite Abelian group $A$. If $A$ has odd order, there exists a bicharacter $\beta:A\times A\to \ku^\times$  such that $\cC \cong \Vc_A^{\beta}$ as braided fusion categories,   where $\Vc_A^{\beta}$ is the braided fusion category of finite-dimensional $G$-graded vector spaces with braiding defined by $c_{\ku_a,\ku_b}=\beta(a,b)\id_{\ku_{ab}}$.

The function \[\bold{Q}:A\to \ku^\times,\quad a\mapsto \beta(a,a),\]is a \emph{quadratic form} on $A$, that is, $\bold{Q}(-a)=\bold{Q}(a)$ and \[\bold{B}(a,a')=\frac{\bold{Q}(a+a')}{\bold{Q}(a)\bold{Q}(a')},\quad \quad a,a'\in A,\] defines a bicharacter. 

Following \cite{DGNO}, we say that a \emph{pre-metric} group is a pair $(A,\bold{Q})$, where $A$ is a finite abelian group and $\bold{Q}:A\to \ku^\times$ is a quadratic form.  A pre-metric group is called \emph{metric} is the associated bicharacter is non-degenerate or equivalently the associated braided fusion category is non-degenerate.

A morphism of pre-metric groups
$(A_1, \bold{Q}_1) \to (A_2, \bold{Q}_2)$ is a homomorphism $\psi: A_1 \to A_2$ such that $\bold{Q}_2(\psi(a)) = \bold{Q}_1(a)$ for
all $a \in A_1.$ The group of all automorphisms of a metric group $(A,\bold{Q})$ will be denoted as $O(A,\bold{Q})$ and called the orthogonal group of   $(A,\bold{Q})$. 

Let  $A$ be an elementary abelian $p$-group, where $p$ is an odd prime. Once a  primitive $p$-th root of unity $\zeta \in \ku^\times$ is fixed, the formula \begin{align*}
    \bold{Q}(a)=\zeta^{Q(a)}, && \forall a\in A,
\end{align*} defines a bijective correspondence between metric groups over $\ku^\times$ of the form $(A,\bold{Q})$ and quadratic spaces $(A,Q)$, where $A$ is seen as a vector space over $\F_p$. Hence, there is a bijective correspondence between non-degenerate pointed braided fusion categories $\cD$ with $\Inv(\cD)= A$ and  quadratic spaces $(A,Q)$ over $\F_p$.

\subsection{Reflection fusion categories} 

Let $(A,Q)$ be a  quadratic space over $\F_p$. We will denote by $\Vc_{(A,Q)}$ the non-degenerate pointed braided fusion category associated to $(A,Q)$. It follows from \cite[Theorem 5.2]{ENO3} that $\Pic{\Vc_{(A,Q)}}= O(A,Q)$.

Let $\cC$ be a fusion category and $S=\{X_i\}_{i\in I}$ a set of objects. We will say $\cC$ is \emph{tensor generated} by  $S$ if for every simple object $Y\in \cC$, there are  $X_1,\ldots X_l\in S$ (possibly with repetitions) such that $Y\subset X_1\otimes \cdots \otimes X_l$. 
\begin{definition}
Let $\cB$ be a  $G$-crossed  fusion category where $\cB_e$ is a non-degenerate pointed braided fusion category with $\Inv(\cB_e)$ an elementary abelian $p$-group. 

We  will say that $\cB$ is a  \emph{reflection fusion category} if 

\begin{enumerate}
    \item\label{condition 1}$\Inv(\cB)=\Inv(\cB_e)$, and
    \item\label{condition 2} $\cB$ is tensor generated by $\Inv(\cB)$ and simple objects of $\operatorname{FP}$-dimension $\sqrt{p}$.
\end{enumerate}
\end{definition}

\begin{lemma}\label{Lemma:dimensions in B_g}
Let $\cB$ be a  $G$-crossed  fusion category where $\cB_e=\Vc_{(A,Q)}$ is a non-degenerate pointed braided fusion category. Let $T_g\in O(A,Q)$ be the isometric automorphism  associated to $\cB_g$. Then all simple objects in $\cB_g$ have  $\operatorname{FP}$-dimension $|\operatorname{Im}(\id_A-T_g)|^{1/2}$.
\end{lemma}
\begin{proof}
Since $\cB_e$ is pointed and $\cB_g$ is an indecomposable module category,  $\Inv(\cB_e)$ acts transitively on $\Irr(\cB_g)$. In particular, all simple objects in $\cB_g$ have the same dimension. Let $X\in \Irr(\cB_g)$ and $\St(X)=\{Y\in \Inv(\cB_e): Y\otimes X\cong X\}$. Since $\FPdim(\cB_g)=\FPdim(\cB_e)$, we have that $\FPdim(X)^2=|\St(X)|$. It follows from \cite[Remark 5.7]{DN}, that $|\St(X)|=|\operatorname{Im}(\id_A-T_g)|$. Then, $\FPdim(X)=|\operatorname{Im}(\id_A-T_g)|^{1/2}$.

\end{proof}

\begin{theorem}\label{Thm:Def reflec cat}
Let $\cB$ be a  $G$-crossed  fusion category where $\cB_e$ is a non-degenerate pointed braided fusion category with $\Inv(\cB_e)$ an elementary abelian $p$-group.  Then $G$ is a  reflection group in $\operatorname{Pic}(\cB_e)$ if and only if $\cB$ is a reflection fusion category.
\end{theorem}
\begin{proof}
Let $\cB_e=\Vc_{(A,Q)}$, where $(A,Q)$ is the associated quadratic space over $\F_p$. The element $T_g\in O(A,Q)=\Pic{\cB_e}$ is trivial if and only if $\cB_g$ contains an invertible object. Then the group homomorphism \[G\to \operatorname{Pic}(\cB_e),\quad  g\mapsto [\cB_g],\] is injective if and only if condition \eqref{condition 1} in the definition of reflection fusion category holds.

Now, let \[S=\{g\in G: T_g\in O(A,Q) \text{ is a reflection}\}.\] We will see that \[S=\{g\in G: \cB_g \text{ contains a simple object with } \FPdim(X)=\sqrt{p}\}.\]  

It follows from Lemma \ref{Lemma:dimensions in B_g} that the objects in $\cB_g$ have $\operatorname{FP}$-dimension $\sqrt{p}$ if and only if $\dim_{\F_p}(\operatorname{Im}(\id_A-T_g))=1$. Since $T_g\in O(A,Q)$, the condition $\dim_{\F_p}(\operatorname{Im}(\id_A-T_g))=1$ is equivalent to being a reflection. In conclusion, $T_g$ is a reflection if and only if $\cB$ has a simple object of $\operatorname{FP}$-dimension $\sqrt{p}$.

We will see  that $G=\langle S\rangle$ if and only if condition \eqref{condition 2} in the definition of reflection fusion category holds. Let $x\in G$ and $Y\in \cB_x$ be a simple object. If condition \eqref{condition 2} holds, there exist $Z\in \cB_e$ and $X_i\in \cB_{g_i}$, with $g_i\in S$, such that $Y\subset Z\otimes X_1\otimes\cdots \ot X_l$. Hence $\cB_{g_1}\otimes \cdots \ot\cB_{g_l}\subset \cB_x$,  and $g_1\cdots g_l=x$. Thus $G=\langle S\rangle$. Conversely, assume that $G=\langle S\rangle$.  Let $x\in G$ and $Y\in \cB_x$ be a simple object. Thus there are $g_i\in S$ such that $x=g_1\ldots g_l$. Let $X_i\in \cB_{g_i}$ be reflection objects and $Y'\subset X_{1}\ot \cdots \ot X_l\in \cB_x$ a simple object. Since $\Inv(\cB_e)$ acts transitively on the isomorphism classes of simple objects of $\cB_x$, there is $Z\in \Inv(\cB_e)$ such that $Y\cong Z\otimes Y'$, hence \[Y\cong Z\ot Y'\subset Z\ot X_1\ot \cdots \ot X_l,\] that is, condition \eqref{condition 2} holds.
\end{proof}

\section{Irreducible reflection fusion categories}
\begin{definition}
A  $G$-crossed fusion category $\cB=\oplus_{g\in G}\cB_g$ will be called \emph{irreducible} if $\cB_e$ does not contain non-trivial $G$-invariant fusion subcategories.
\end{definition}

\begin{proposition}

Let $\cB$ be an irreducible  $G$-crossed category with $G\ne 1$. The following conditions are equivalent: 
\begin{enumerate}
    \item $\FPdim(\cB_e)=p^n$ ($p$ an odd prime), $\Inv(\cB)=\Inv(\cB_e)$, and $\cB$ is generated by objects of $\operatorname{FP}$-dimension $1$ and $\sqrt{p}$.
    \item  $\cB$ is a reflection category.
\end{enumerate}

\end{proposition}
\begin{proof}
(2) $\implies $(1) is trivial, for (1) $\implies$ (2) we  only need to see that $\cB_e$ is a non-degenerate pointed fusion category with $\Inv(\cB_e)$ an elementary abelian $p$-group.

As a consequence of \cite[Theorem 8.28]{ENO},  we have that $\Inv(\cB_e)$ is non-trivial.  Since $\cB$ is irreducible, we have that $\Irr(\cB_e)=\Inv(\cB_e)$, with  $\Inv(\cB_e)$ an elementary abelian $p$-group, where $p$ is an odd prime. Since the fusion subcategory generated by all $X\in \cB_e$ such that $c_{Y,X}c_{X,Y}=\id_{X\ot Y}$ is invariant for any braided auto-equivalence, we have that $\cB_e$ is either non-degenerate or symmetric, with $p$ invertible objects. But in this case $\cB$ $\cB$ will have invertible objects outside $\cB_e$ unless $G=1$, since the category of vector spaces is not invertible as a module category over $\cB_e$.  
\end{proof}

The aim of this section is to classify irreducible reflection fusion categories up to twist using the results of Appendix \ref{Appendix:Etingof} and the classification of irreducible reflection groups over finite fields, \cite{MR603578}.

If $\cB=\oplus_{g\in G}\cB_g$ is an irreducible reflection fusion category with $\cB_e=\Vc_{(A,Q)}$, it follows from Theorem \ref{Thm:Def reflec cat} and \cite[Theorem 7.12]{ENO3} that $G\subset O(A,Q)=\Pic{\Vc_{(A,Q)}}$ is an irreducible orthogonal reflection group. 

Conversely, it follows from Theorem \ref{main2} that in order to classify all irreducible  $G$-crossed extensions of $\Vc_{(A,Q)}$ up to twisting, we can proceed as follows:

\begin{enumerate}
    \item Compute all orthogonal irreducible reflection groups $G\subset O(A,Q)$.
    \item Compute $H^2(G,A)$ and all elements $L \in H^2(G,A)$ such that $\beta(L,L)=1\in H^4(G,\ku^\times)$, where $\beta(L,L)$ is the composition of the bicharacter $\beta$ associated with $Q$ with the cup product in the group cohomology of $G$.
\end{enumerate}For each such pair $(G,L)$ there is then a unique $G$-crossed extension up to twisting by an element of $H^3(G,\ku^\times)$. 

Using that $H^n(G,A)=0$ for all $n\geq 1$ if $(|G|,|A|)=1$ and Theorem \ref{main2}, we obtain:

\begin{corollary}\label{cor:non-modular case}
Let $(A,Q)$ be a  quadratic space over $\F_p$ and $G\subset O(A,Q)$ a reflection group where $H^2(G,A)=0$. Then up to twisting there is a unique reflection fusion category
 $\cB=\oplus_{g\in G}\cB_g$ where $\cB_e=\Vc_{(A,Q)}$ associated to $G$. In particular, this is so  if $p\nmid |G|$.
\end{corollary}
\qed

\subsection{Irreducible reflection fusion categories where $p \nmid |G|$}

The case $G=1$ is trivial, so we will omit it and assume $G\ne 1$. If $\cB=\oplus_{g\in G}\cB_g$ is an irreducible reflection category with $\Inv(\cB_e)=\F_p$, then $G=\Z/2\Z$  and $G$ acts on $\F_p$ by inversion (since $\Pic{\cB_e}=\{\id, -\id\}$). It follows from Lemma \ref{Lemma:dimensions in B_g}  that $\cB_{1}$ has only one simple object up to isomorphism. Hence $\cB$ is a Tambara-Yamagami fusion category (see \cite{TY}) with a  $\Z/2\Z$-crossed structure (see \cite[Section 5]{GNN} for a complete description).

The next theorem classifies irreducible reflection fusion categories $\cB=\oplus_{g\in G}\cB_g$, where  $\Inv(\cB_e)$ is $p^n$ and $p \nmid |G|$.

\begin{theorem}\label{Theorem: classification relfection G-invariant bilineat form}
Let $p$ be an odd prime. Every  irreducible reflection fusion category $\cB=\oplus_{g\in G}\cB_g$  with $|\Inv(\cB_e)|=p^n$ and $p\nmid |G|$ is a reflection category of Coxeter type, namely its reflection group is a twisted form of 

\begin{table}[!htbp]

\begin{tabular}{ |p{2cm}|p{2cm}|p{6cm}|  }
 \hline
 $G$ & $\operatorname{dim}(V)$ &  Conditions on $p$ \\
 \hline
 $A_{n,p}$  &  $n\geq 1$   & $p\not{|} \ (n+1)!$ \\
 \hline
   $B_{n,p}$  &  $n\geq 3$   & $p\not{|}\  n!$\\
 \hline
   $D_{n,p}$  &  $n\geq 4$   & $p\not{|}\  n!$\\
 \hline
    $E_{6,p}$  &  $6$   & $p>5$\\
 \hline
   $E_{7,p}$  &  $7$   & $p>7$ \\
 \hline
    $E_{8,p}$  &  $8$   & $p>7$\\
 \hline
    $F_{4,p}$  &  $4$   & $p>3$\\
 \hline
     $H_{3,p}^\zeta$    &  $3$   & $p>5$, $p^2-1\equiv 0 \mod 5$ \\
 \hline    $H_{4,p}^\zeta$  &  $4$   &  $p>5$, $p^2-1\equiv 0 \mod 5$ \\
 \hline   $I_{2,p}(d)$  &  $2$   &    $d| p-1  $ if $I_{2,p}(d)\subset O^{+}(2,p)$, and $d|p+1$  if $I_{2,p}(d)\subset O^{-}(2,p)$. \\
 \hline
\end{tabular}

\caption {Irreducible  reflection groups over $\Fp$ with $p \nmid |G|.$} \label{tab:coxeter non-modular}
\end{table}

\end{theorem}
\begin{proof}
The theorem follows immediately from Corollary \ref{cor:non-modular case} and Theo\-rem \ref{Th:classification irr othogonal reflection groups}.
\end{proof}
\begin{remark}
Since $A_3=D_3$, we have excluded $D_3$ from Table \ref{tab:coxeter non-modular}.
\end{remark}

\subsection{Reflection fusion categories where $p | \ |G|$}

As we saw in Coro\-llary \ref{cor:non-modular case}, reflection fusion categories where  $p$ does not divide the order of $G$ are classified up to twisting by an element of $H^3(G,\ku^\times)$ by an orthogonal reflection group  $G\subset O(V,Q)\cong \Pic{\cB_e}$. As we will see, even if $p$ divides $|G|$, in almost all cases of irreducible reflection groups we have that $H^2(G,V)=0$ and thus Corollary \ref{cor:non-modular case} applies.

The following lemma will be useful to compute $H^2(G,V)$.
\begin{lemma}\label{Lemma:LHS}
Let $G$ be a group and $A$ a $G$-module. If there is a normal subgroup $N\subset G$ such $H^i(N,A)=0$ for all $i\leq n$ then $H^n(G, A)=0$.
\end{lemma}
\begin{proof}
Let us use the short exact sequence \[1\to N\to G\to G/N\to 1.\] The associated Lyndon–Hochschild–Serre spectral sequence has $E_2$-page \[E^{p,q}_2 = H^p(G/N,H^q(N,A)),\] and it converges to the group cohomology $E^n_\infty = H^n(G,A)$. Since  $H^q(N, A)=0$ for all $q\leq n$, we have that  $E^{p,q}_2 =0$ for all $q\leq n$. Thus $0=E^n_\infty=H^n(G,A)$.
\end{proof}
\begin{lemma}\label{Lemma:H^2 =0}
Let $G\subset O(V,Q)$ be an irreducible orthogonal reflection group over $\F_p$. Thus $H^2(G,V)=0$, except in the following cases:  $O_2(3,p)$ with $p>3$, $O_2(5,3)$, $O_2(7,3)$ and  $O^{-}_2(8,3)$.
\end{lemma}
\begin{proof}
We will start case by case with the groups of Theorem  \ref{Th:classification irr othogonal reflection groups}.

(a) \emph{Reduction mod $p$ of finite Coxeter groups.} For the case $A_{n-1,p}$, we have that $V$ is the irreducible Specht module associated to the partition $(n-1,1)$. Thus by \cite[Theorem 4.1]{BURICHENKO}, $H^2(\mathbb{S}_n,V)=0$.

If $G$ is one of the following groups $B_n, E_7, E_8$ and $F_4$, we have that $-\id_V\in G$. Hence $H^2(G,V)=0$  by Lemma \ref{Lemma:LHS}.

In the $D_{n,p}$ case, the reflection group has the form $A\rtimes \mathbb{S}_n,$ where $A$ is the group of automorphisms $u$ of the vector space $V=\F_p^n$ such that  $u(e_i)=(-1)^{\nu_i}e_i$ with $\prod_{i=1}^n(-1)^{\nu_i}=1$. The group $A$ is isomorphic to $(\Z/2\Z)^{n-1}$. Hence again \[H^0(A,V)=V^{A}=0 \text{ and }H^i(A,V)=0,\] for all $i>0$. It follows from Lemma \ref{Lemma:LHS} that $H^2(A\rtimes \mathbb{S}_n,V)=0$.

If $p>5$ and $G=E_{6,p}$, we have that $H^2(G,V)=0$ since $(|G|,|V|)=0$. For $G=E_{6,5}$, consider  $\mathbb{S}_5\subset G$, that corresponds to the copy of $A_4$ in $E_6$ which is spanned by the roots $\epsilon_i-\epsilon_{i+1},$  $i=3,4,5,6,7$, in \cite[page 94]{Kane-Book} presentation of $E_6$. Let $g\in \mathbb{S}_5$ be an element of order 5. Hence the inclusions $\Z/5\Z\cong \langle g\rangle \subset \mathbb{S}_5\subset E_{6,5}$, define the group homomorphisms 

\[
\xymatrix{
H^2(E_{6,5},V) \ar[dr]_{j} \ar[rr]^{f}&& H^2(\Z/5\Z,V) \\
&H^2(\mathbb{S}_5,V), \ar[ru]_{h}& 
}
\] 
where $f$ and $h$ are injective ( $\langle g\rangle$ is 5-Sylow subgroup and $V$ is a $5$-group). Note that $V|_{\mathbb{S}_5}=P\oplus \F_5$, where $P$ is the permutation representation and $\F_5$ is the trivial representation. Now, by Shapiro's lemma, $H^2(\mathbb{S}_5,P)=H^2(\mathbb{S}_4, \F_5)=0$, and  $H^2(\mathbb{S}_5, \F_5)=0$ as well (see \cite{BURICHENKO}). Thus $H^2(\mathbb{S}_5,V)=0$, so $f$ is an injective map which factors through the zero space, that is, $f=0$ and then $H^2(E_{6,5},V)=0$.

(b) \emph{$\overline{A}_{n-2,p}\cong \mathbb{S}_{n}$, with $p| n$ and $n>4$.} Following the notation in \cite{BURICHENKO}, we have that $V$ is the irreducible module $D^{(n-1,1)}$. Thus, it follows from \cite[Proposition 5.4]{BURICHENKO} that $H^2(\mathbb{S}_n,V)=0$ for all $n>3$.

(c) \emph{Groups of orthogonal type.} Since  $-\id_V\in O_1(V,Q)\subset O(V,Q)$  and $V$ has odd order, using Lemma \ref{Lemma:LHS} we have that $H^2(O(V,Q),V)=H^2(O_1(V,Q),V)=0$.

By \cite[Section 3.1 ]{MR2074400}, we have  that excluding the cases $O(3,3)$ and $O^{+}(4,3)$,  the derived subgroup $\Omega(V,Q)$ of $O(V,Q)$, is a normal subgroup of $O_2(V,Q)$ of index two. Hence we have an exact sequence \[1\to \Omega(V,Q)\to O_2(V,Q)\to \Z/2\Z\to 1.\]

It follows from \cite{Kusefoglu-1,Kusefoglu-2} and \cite{Burichenko-english,Burichenko-russian} 
\footnote{According to \cite{Kusefoglu-1,Kusefoglu-2}, $H^2(\Omega(V,Q),V)=0$ for $\dim(V)>3$, except in the cases $\Omega(V,Q)\subset O(7,3)$ and  $\Omega(V,Q)\subset O^{-}(8,3)$. However, according to \cite{Burichenko-english,Burichenko-russian}, the papers \cite{Kusefoglu-1,Kusefoglu-2} had some mistakes. In particular,  $H^2(\Omega(V,Q),V)\neq 0$ in two more cases:  $\Omega(V,Q)\subset O^{-}(4,9)$ and $\Omega(V,Q)\subset O(5,3)$.} 
that $H^2(\Omega(V,Q),V)=0$ except for the groups $\Omega(V,Q)\subset O(3,p)$ with $p>3$, $\Omega(V,Q)\subset O(5,3)$, $\Omega(V,Q)\subset O(7,3)$ and  $\Omega(V,Q)\subset O^{-}(8,3)$. Since $H^0(\Omega(V,Q),V)=0$ and $H^1(\Z/2\Z,H^1(\Omega(V,Q),V))=0$, using the Lyndon–Hochschild–Serre spectral sequence as in Lemma \ref{Lemma:LHS} we can conclude that $H^2(G,V)=0$, for every $O_i(V,Q)$, different from $O_2(3,p)$ with $p>3$, $O_2(5,3)$, $O_2(7,3)$ and  $O_2^{-}(8,3)$.

\end{proof}

\begin{lemma}\label{lemma: cup product}
We have $H^2(O_2(3,p),V)=\F_p$ for all $p>3$ and the squaring map $H^2(O_2(3,p),V)\to H^4(O_2(3,p),\F_p)$ is zero for all $p> 3$.
\end{lemma}
\begin{proof}
First we need the well known fact about the cohomology of finite cyclic groups.
\begin{sublemma}\label{sublemma}
Let $C_m$ be the cyclic group of order $m$ generated by $\sigma$ and $A$ a $C_m$-module. Then \begin{align}\label{eq:coho_cyclic} H^n(C_m, A)= \begin{cases}\ker(N)/\operatorname{Im}(1-\sigma), \qquad &\text{if } n=1, 3, 5, \ldots \\
\Ker(1-\sigma)/\operatorname{Im}(N), \quad  &\text{ if } n = 2, 4, 6, \ldots.
\end{cases}
\end{align}where  $N= 1+ \sigma + \sigma^2 +\cdots +\sigma ^{m-1}$
\end{sublemma}
\begin{proof}[Proof of Sublemma \ref{sublemma}]
Let $C_m$ be the cyclic group of order $m$ generated by $\sigma$ and $A$ a $C_m$-module. The sequence
$$\begin{diagram}\node{\cdots}\arrow{e,t}{N} \node{\mathbb Z
C_m}\arrow{e,t}{\sigma -1}\node{\mathbb Z
C_m}\arrow{e,t}{N}\node{\mathbb Z C_m}\arrow{e,t}{\sigma
-1}\node{\mathbb Z C_m}\arrow{e}\node{\mathbb Z}\end{diagram}$$
where  $N= 1+ \sigma + \sigma^2 +\cdots +\sigma ^{m-1}$ is a free resolution of $\mathbb Z$ as $C_m$-module. Using that $H^n(C_m,A)= \operatorname{Ext}_{\mathbb{Z}[C_m]}(\Z, A)$, we prove the sublemma.
\end{proof}
Let $\Omega(3,p)$ be the derived subgroup of $O(3,p)$.  It follows \cite{SAH1974255} that $H^2(\Omega(3,p),V)=0$ for $p=3$, and $H^2(\Omega(3,p),V)=\F_p$ for $p>3$, where $V=\F_p^3$ is the standard representation. Using the exact sequence 
\[1\to \Omega(3,p)\to O_2(3,p)\to \Z/2\Z\to 1,\]we can see that $H^2(O_2(3,p),V)=H^2(\Omega(3,p),V)=\F_p$.

Let $G=O_2(3,p)$. It suffices to show that the cup product combined with inner product in $V$ defines the zero map 
\[H^2(G,V)\otimes H^2(G,V)\to H^4(G,\F_p).\]

The $p$-Sylow subgroup of $G$ is isomorphic to $\Z/p\Z$, and we have $H^2(\Z/p\Z,\F_p)=\F_p$. Moreover, it is well known that the restriction morphism $H^*(G,V)\to H^*(S,V)$ is injective if $S$ is $p$-Sylow in $G$ and $V$ any $G$-module of exponent $p$ (see  \cite[Corollary 4.2.11]{gille_szamuely_2006}). Hence the restriction map $H^2(G,V)\to H^2(\Z/p\Z,V)$ is an isomorphism, and it suffices to show that the natural bilinear map 
$H^2(\Z/p\Z,V)\otimes H^2(\Z/p\Z,V)\to H^4(\Z/p\Z,\F_p)$
 is zero. 

Let $g\in O_2(3,p)$ be an element of order $p$. Since $(\id_V-g)^p=0$,  $\id_V-g$ is nilpotent. Hence $\id_V+g+\cdots +g^{p-1}=(\id_V-g)^{p-1}=0$,  since $p-1>3$. It follows from \eqref{eq:coho_cyclic} that  $H^2(G,V)=V^g=\F_p$ (the trivial module). Thus the natural map $H^2(\Z/p\Z,V^g)\to H^2(\Z/p\Z,V)$ is an isomorphism. But the inner product on $V$ vanishes when restricted to $V^g$. This implies the statement. 
\end{proof}
\begin{theorem}
Any irreducible reflection group over $\F_p$ with $p\geq 3$ gives rise to a unique reflection category up to twisting, except possibly for the cases indicated in the Table \ref{tab:H^2 neq 0}. If $G=O_2(3,p)$, the set of reflection fusion categories up to twisting is a non-empty torsor over $\F_p$.

\begin{table}[!htbp]

\begin{tabular}{ |p{2cm}||p{3cm}|p{2cm} |p{2cm}|p{2cm}| }
 \hline
 $G$ & $O_2(3,p)$, $p>3$ & $O_2(5,3)$ & $O_2(7,3)$ & $O_2^{-}(8,3)$\\
 \hline
  $H^2(G,V)$ & $\F_p$ & $\F_3$ & $\F_3$ & $\F_3^2$\\
  \hline
\end{tabular}
\caption {Irreducible reflection groups over $\Fp$ with non-trivial $H^2(G,V)$.} \label{tab:H^2 neq 0}
\end{table}

\begin{table}[!htbp]

\begin{tabular}{ |p{2.7cm}|p{2cm}|p{6cm}|  }
 \hline
 $G$ & $\operatorname{dim}_{\F_p}(V)$ &  Restriction on $p>2$ or $n$ \\
 \hline
 $A_{n,p}$  &  $n>0 $   & $p\not{|} \ n+1$ \\
 \hline
   $B_{n,p}$  &  $n>2$   & \\
 \hline
   $D_{n,p}$  &  $n>3$   & \\
 \hline
    $E_{6,p}$  &  $6$   & $p\geq 5$\\
 \hline
   $E_{7,p}$  &  $7$   &  \\
 \hline
    $E_{8,p}$  &  $8$   & \\
 \hline
    $F_{4,p}$  &  $4$   & \\
     \hline   $I_{2,p}(d)$  &  $2$   &  $d| p-1  $ if $I_{2,p}(d)\subset O^{+}(2,p)$, and $d|p+1$  if $I_{2,p}(d)\subset O^{-}(2,p)$. \\
 \hline
     $H_{3,p}^\zeta$    &  $3$   & $p>5$, $p^2-1\equiv 0 \mod 5$ \\
 \hline    $H_{4,p}^\zeta$  &  $4$   &  $p>5$, $p^2-1\equiv 0 \mod 5$ \\
  \hline
\hfill \break $\overline{A}_{n,p}$ &  $n\geq 3$   &  $p | n+1$ \\ 
 \hline
   $O^{\pm}(2n,p)$  &  $2n$   &   \\
 \hline
$O(2n+1,p)$  &  $2n+1$   &   \\

 \hline
$O(2n+1,p)$  &  $2n+1$   &   \\
 \hline
    $O_{i}^{\pm}(2n,p)$, \newline ($i=1,2$)  &  $2n$   &   If $p=3$ then $n\neq 8$. \\
 \hline
    $O^{\pm}_i(2n+1,p)$, \newline ($i=1,2$)  &  $2n+1$   &   If $i=2$ and $p=3$, then $n\neq 5, 7$\\
 \hline
\end{tabular}

\caption {Irreducible reflection fusion categories over $\Fp$ up to twisting by elements in $H^3(G,\ku^\times)$.} \label{tab:irr reflec fusion cat}
\end{table}
\end{theorem}
\begin{proof}
It follows from Theorem \ref{Thm:Def reflec cat} that  $G\subset O(V,Q)$ is an irreducible orthogonal reflection group. Hence $G$ is one of the groups of Theorem \ref{Th:classification irr othogonal reflection groups} or a Tambara-Yamagami fusion category. 

By Lemma \ref{Lemma:H^2 =0} and Lemma \ref{lemma: cup product}, we have that $H^2(G,V)=0$, except for $O_{2}(3,p)$ ($p>3$), where $H^2(O_{i}(V,Q),V)=\F_p$. It follows from Corollary \ref{cor:non-modular case} that in all cases where $H^2(G,V)=0$, $\cB$ is unique up to twisting by elements in $H^3(G,\ku^\times).$

Let $G=O_2(3,p)$. By Lemma \ref{lemma: cup product}, for any $M\in H^2(G,V)=\F_p$ the fourth obstruction $O_4(M)$ is zero. Hence the set  of twisting classes of reflection categories $\cB=\oplus_{g\in G}\cB_g$ over $\Vc_{(V,Q)}$, is a non-empty torsor over $\F_p$.
\end{proof}

\begin{remark} The condition that ${\rm Inv}(\cB)={\rm Inv}(\cB_e)$ in the definition of a reflection category 
can be relaxed without significantly changing the theory. Namely, let $\cB$ be a crossed $G$-category with 
$\cB_e$ pointed and ${\rm Inv}(\cB_e)=V$ an elementary abelian $p$-group. Let us call $\cB$ a {\it generalized reflection category} if it is tensor generated by objects of Frobenius-Perron dimension $1$ and $\sqrt{p}$. Then the theory of generalized reflection categories is the same as that of reflection categories, except that the corresponding map 
$G\to {\rm Pic}(\cB)=O(V,Q)$ does not have to be injective, so the corresponding reflection group is not $G$ but the image $\overline{G}$ of $G$ in $O(V,Q)$. Of course, in this situation we can easily have a nontrivial squaring map $\beta: H^2(G,V)\to H^4(G,\ku^\times)$, so the classification of irreducible generalized reflection categories is somewhat less explicit than that of reflection categories: up to twisting, they are parametrized by $L\in H^2(G,V)$ such that $\beta(L,L)=1$. 
\end{remark}

\section{ Appendix: Classification of extensions of pointed fusion categories associated to elementary abelian $p$-groups, $p>2$. \\ By Pavel Etingof}\label{Appendix:Etingof}

\subsection{Introduction and the main results}

Let $\ku$ be an algebraically closed field of characteristic zero, $A$ be a finite abelian group, 
$A^*:=\Hom(A,\ku^\times)$, and $\Vec_A$ be the fusion category of $A$-graded $\ku$-vector spaces. Recall that by \cite[Corollary 1.2]{ENO}, the Brauer-Picard group ${\rm BrPic}(\Vec_A)$ is naturally isomorphic to the orthogonal group $O(A\oplus A^*,q)$ of automorphisms of $A\oplus A^*$ preserving the quadratic form $q(a,f)=f(a)$. 

Assume now that $A$ has odd order. Let $G$ be a group (not necessarily finite) and $c: G\to O(A\oplus A^*,q)$ be a homomorphism. Then $q$ defines a unique $G$-invariant symmetric bicharacter $\beta$ on $A\oplus A^*$ valued in roots of unity of odd order such that 
\begin{equation}\label{form1}
\beta(x,y)^2=q(x+y)/q(x)q(y);
\end{equation}
namely, $\beta((a_1,f_1),(a_2,f_2))=(f_1(a_2)f_2(a_1))^{1/2}$. Hence $c$ defines a canonical action of $G$ on the braided tensor category ${\mathcal Z}(\Vec_A)=(\Vec_{A\oplus A^*},\beta)$ (the Drinfeld center of $\Vec_A$). In other words, the first obstruction $O_3(c)\in H^3(G,A\oplus A^*)$ to constructing an extension of $\Vec_A$ by $G$ automatically vanishes, and there is a canonical choice of the element $M$ in the corresponding $H^2(G,A\oplus A^*)$-torsor to go to the next step of the process (see \cite[Sections 7,8]{ENO}). This also follows from the fact that $H^i(O(A\oplus A^*,q),A\oplus A^*)=0$ for $i>0$, since $-{\rm Id}\in O(A\oplus A^*)$ acts by $-1$ on $A\oplus A^*$.  

Recall also that if $\cC$ is a faithful $G$-graded tensor category over $\ku$ and $\omega\in H^3(G,\ku^\times)$ then one can define a faithful $G$-graded tensor category $\cC^\omega$ obtained by twisting $\cC$ by $\omega$. This is done by multiplying 
the associativity morphism $\alpha_{X_1,X_2,X_3}$ by $\widetilde\omega(g_1,g_2,g_3)$ for any $g_1,g_2,g_3\in G$ and $X_i\in \cC_{g_i}$, $i=1,2,3$, where $\widetilde\omega$ is a 3-cocycle representing $\omega$. Note that $\cC^\omega$ does not depend on the choice of $\widetilde{\omega}$ and uniquely determines $\omega$. 

Our first main result is the following theorem. 

\begin{theorem}\label{main1} Let $p$ be an odd prime and $A$ be an elementary abelian $p$-group (i.e., a vector space over $\F_p$). Then 

(i) the second obstruction $O_4(c,M)\in H^4(G,\ku^\times)$ vanishes. In other words, there is a canonical $G$-extension $\mathcal{C}$ of $\Vec_A$ induced by $c$, defined up to twisting by an element $\omega\in H^3(G,\ku^\times)$;

(ii) $G$-extensions of $\Vec_A$, up to twisting by $\omega\in H^3(G,\ku^\times)$, are in natural bijection with elements 
$L\in H^2(G,A\oplus A^*)$ such that $\beta(L,L)\in H^4(G,\ku^\times)$ vanishes.  
\end{theorem} 

Here $\beta(L,L)$ denotes the composition of the bicharacter $\beta$ with the cup product in the cohomology of $G$. 

The second main result concerns  $G$-crossed categories. Namely, let $(B,q)$ be a finite abelian group of odd order with a non-degenerate quadratic form $q$, and let 
$\Vec_{(B,q)}$ be the corresponding pointed braided fusion category over $\ku$ via the Joyal-Street theorem (see \cite[Subsection 8.4]{EGNO}). 
Then, similarly to the above, we have a canonical bicharacter $\beta$ on $B$ satisfying \eqref{form1}. Recall (see e.g. \cite{DN}, Theorem 4.3 or Proposition 5.14, Example 5.16) that ${\rm Pic}(\Vec_{(B,q)})=O(B,q)$, the orthogonal group of $(B,q)$.
Thus, for any group $G$ and a homomorphism $c: G\to O(B,q)$ we have a canonical action of $G$ on the braided category $\Vec_{(B,q)}$. That is, the obstruction $O_3(c)\in H^3(G,B)$ to constructing a $G$-crossed extension of $B$ vanishes, and there is a canonical choice of $M$ in the corresponding torsor over $H^2(G,B)$ (see \cite{ENO}, Subsections 7.6, 7.8). 
The second obstruction $O_4(c,M)$ to constructing a $G$-crossed extension is then an element of $H^4(G,\ku^\times)$
(see \cite{ENO}, Subsection 7.8). 

Recall also that similarly to usual $G$-graded extensions, $G$-crossed extensions of $\Vec_{B,q}$ form a torsor over $H^3(G,\ku^\times)$, i.e., can be twisted an element $\omega$ of this group, $\mathcal{B}\mapsto \mathcal{B}^\omega$. 

Our second main result is the following theorem. 

\begin{theorem}\label{main2} Let $p$ be an odd prime and $B$ be an elementary abelian $p$-group. Then 

(i) the second obstruction $O_4(c,M)\in H^4(G,\ku^\times)$ vanishes. In other words, there is a canonical braided $G$-crossed extension $\mathcal{B}$ of $\Vec_{(B,q)}$ induced by $c$, defined up to twisting by an element $\omega\in H^3(G,\ku^\times)$;

(ii)  $G$-crossed extensions of $\Vec_{(B,q)}$, up to twisting by $\omega\in H^3(G,\ku^\times)$, are in natural bijection with elements 
$L\in H^2(G,B)$ such that $\beta(L,L)\in H^4(G,\ku^\times)$ vanishes.

\end{theorem} 

Note that Theorem \ref{main1} is a special case of Theorem \ref{main2}, since the problem of constructing a $G$-extention of $\Vec_A$ is equivalent to the problem of constructing  
a  $G$-crossed extension of the Drinfeld center of $\Vec_A$.\footnote{I am grateful to D. Nikshych for this remark.} 

The proof of Theorems \ref{main1}, \ref{main2} is based on the result of Fiedorowicz and Priddy \cite{FP} on the vanishing of the stable cohomology of orthogonal groups over a finite field with coefficients in a field of the same characteristic. 
However, we expect that these theorems can also be proved by directly computing the obstruction $O_4$ without using the results of \cite{FP}, and that they actually hold for any finite abelian group $A$ of odd order. 

{\bf Acknowledgements.} I am grateful to D. Nikshych and V. Ostrik for useful discussions. This work was supported by the NSF grant DMS-1502244.

\subsection{Proofs of Theorems \ref{main1} and \ref{main2}}

\subsubsection{Proof of Theorem \ref{main1}}

 Let $O(2n,p)$ be the split orthogonal group over $\F_p$ in dimension $n$. Also let $\ku^\times_p$ be the group of elements of $\ku^\times$ whose order is a power of $p$. The proof of Theorem \ref{main1} is based on the following proposition, which is an immediate consequence of 
\cite{FP}, Theorem 4.6(d).   

\begin{proposition}\label{mainprop} Let $i>0$. If $p$ is an odd prime then 
there exists $n_0$ such that for $n\ge n_0$ one has $H^i(O(2n,p),\ku^\times_p)=0$. 
\end{proposition} 

Now we can prove Theorem \ref{main1}(i). Let $A=\F_p^m$, and for $n\ge m$, let $c_n: G\to O(2n,p)$ 
be the composition of $c=c_m$ with the natural inclusion $O(2m,p)\hookrightarrow O(2n,p)$ (induced by replacing $A$ by 
$A\oplus \F_p^{n-m}$). Similarly, let $M_n$ be the canonical choice 
of the element of the torsor over $H^2(G, \F_p^{2n})$ (so $M_m=M$).
This choice determines the fusion ring $R_n$ of the extended category, 
and $R_n=R\otimes \Z[\F_p^{n-m}]$, where $R=R_m$ (with the natural basis of the tensor product). 
Thus $O_4(c,M)$ vanishes if an only if so does $O_4(c_n,M_n)$. Indeed, any categorification 
of $R$ defines a categorification of $R_n$ (by tensoring with $\Vec_{\F_p^{n-m}}$), and 
vice versa. But $O_4(c_n,M_n)$ is pulled back from an element  of $H^4(O(2n,p),\ku^\times)$, which by Theorem 8.16 of \cite{ENO} in fact lies in 
$H^4(O(2n,p),\ku_p^\times)$. Also, by Proposition \ref{mainprop}, the group $H^4(O(2n,p),\ku_p^\times)$ is trivial for sufficiently large $n$. 
This implies that $O_4(c_n,M_n)$ vanishes for sufficiently large $n$, hence $O_4(c,M)$ vanishes, as desired. 

Part (ii) of Theorem \ref{main1} now follows from \cite{ENO}, Proposition 8.13 and Proposition 8.15. 

\subsubsection{Proof of Theorem \ref{main2}}

Theorem \ref{main2} follows from Theorem \ref{main1}. 
Namely, the obstruction $O_4(c,M)$ to constructing a  $G$-crossed extension 
of $\Vec_{(B,q)}$ is the same as the obstruction $O_4(c,M)$ to constructing a usual $G$-extension of $\Vec_B$,
which vanishes by Theorem \ref{main1}(i). This proves Theorem \ref{main2}(i).  

Part (ii) of Theorem \ref{main2} now follows in the same way as part (ii) of Theorem \ref{main1}.

\end{document}